\documentclass[reqno,12pt]{amsart}

\usepackage{eucal, amscd, amstext,  enumerate}
\usepackage{amsmath, amsthm, amscd, amsfonts, amssymb, graphicx, color}
\usepackage[bookmarksnumbered, colorlinks, plainpages]{hyperref}

\newtheorem{theorem}{Theorem}[section]
\newtheorem{lemma}[theorem]{Lemma}
\newtheorem{proposition}[theorem]{Proposition}

\theoremstyle{definition}

\newtheorem{example}[theorem]{Example}

\theoremstyle{remark}

\numberwithin{equation}{section}

\DeclareSymbolFont{AMSb}{U}{msb}{m}{n}

\DeclareMathSymbol{\A}{\mathbin}{AMSb}{"41}
\DeclareMathSymbol{\B}{\mathbin}{AMSb}{"42}
\DeclareMathSymbol{\C}{\mathbin}{AMSb}{"43}
\DeclareMathSymbol{\D}{\mathbin}{AMSb}{"44}
\DeclareMathSymbol{\E}{\mathbin}{AMSb}{"45}
\DeclareMathSymbol{\F}{\mathbin}{AMSb}{"46}
\DeclareMathSymbol{\G}{\mathbin}{AMSb}{"47}

\DeclareMathSymbol{\HH}{\mathbin}{AMSb}{"48}

\DeclareMathSymbol{\I}{\mathbin}{AMSb}{"49}

\DeclareMathSymbol{\N}{\mathbin}{AMSb}{"4E}

\DeclareMathSymbol{\PP}{\mathbin}{AMSb}{"50}

\DeclareMathSymbol{\Q}{\mathbin}{AMSb}{"51}
\DeclareMathSymbol{\R}{\mathbin}{AMSb}{"52}

\DeclareMathSymbol{\SSS}{\mathbin}{AMSb}{"53}

\DeclareMathSymbol{\T}{\mathbin}{AMSb}{"54}
\DeclareMathSymbol{\U}{\mathbin}{AMSb}{"55}
\DeclareMathSymbol{\V}{\mathbin}{AMSb}{"56}
\DeclareMathSymbol{\W}{\mathbin}{AMSb}{"57}
\DeclareMathSymbol{\X}{\mathbin}{AMSb}{"58}
\DeclareMathSymbol{\Y}{\mathbin}{AMSb}{"59}
\DeclareMathSymbol{\Z}{\mathbin}{AMSb}{"5A}

\def\th#1{\bigskip\noindent{\bf #1}\bgroup\it}
\def\tth#1{\bigskip\noindent{\bf #1}\bgroup}
\def\endth{\egroup\par\bigskip}
\def\endtth{\egroup\par\bigskip}

\newcommand{\vp}{\varphi}

\begin{document}

\title
[Holomorphic functions of matrices]
{Orthogonally additive and orthogonally multiplicative holomorphic functions of matrices}

\author[Q. Bu, C.-J. Liao and N.-C. Wong]{Qingying Bu$^1$, Chingjou Liao$^2$ and Ngai-Ching Wong$^3$$^{*}$}

\address{$^1$ Department of Mathematics, University of Mississippi, University, MS 38677, USA}
\email{\textcolor[rgb]{0.00,0.00,0.84}{qbu@olemiss.edu}}

\address{$^2$ Department of  Mathematics,
Hong Kong Baptist University, Hong Kong.}
\email{\textcolor[rgb]{0.00,0.00,0.84}{cjliao@hkbu.edu.hk}}

\address{$^3$ Department of Applied Mathematics,
National Sun Yat-sen University, Kaohsiung, 80424, Taiwan.}
\email{\textcolor[rgb]{0.00,0.00,0.84}{wong@math.nsysu.edu.tw}}

\dedicatory{This paper is dedicated to Professor Tsuyoshi Ando}




\date{Received: xxxxxx; Revised: yyyyyy; Accepted: zzzzzz.
\newline \indent $^{*}$ Corresponding author}

\begin{abstract}
Let $H:M_m\to M_m$ be a holomorphic function of the algebra $M_m$ of  complex $m\times m$ matrices.
Suppose that $H$ is  orthogonally additive
and orthogonally multiplicative on self-adjoint elements.
We show that either the range of $H$ consists of zero trace elements, or there is
a scalar sequence $\{\lambda_n\}$ and an invertible $S$ in
$M_m$ such that
$$
H(x) =\sum_{n\geq 1} \lambda_n S^{-1}x^nS, \quad\forall x \in M_m,
$$
or
$$
H(x) =\sum_{n\geq 1} \lambda_n S^{-1}(x^t)^nS, \quad\forall x \in M_m.
$$
Here, $x^t$ is the transpose of the matrix $x$.
In the latter case,
we always have the first representation form when $H$ also preserves zero products.
We also discuss the cases where the domain and the range carry different dimensions.
\end{abstract}

\maketitle

\section{Introduction}

Let $E$ and $F$ be real or complex Banach spaces,
and $n$ a positive integer. A map $P: E \to F$ is called a \emph{bounded
$n$-homogeneous polynomial} if there is a  bounded  symmetric  $n$-linear
operator $T: E\times\cdots\times E \to F$ such that
$$
P(x) = T(x,\ldots, x), \quad\forall x \in E.
$$
In this case, we have
$$
T(x_1,\ldots, x_n) = \frac{1}{2^nn!}\sum_{\epsilon_i=\pm1}
\epsilon_1\cdots\epsilon_n P\left(\sum_{i=1}^n \epsilon_ix_i\right),
\quad\forall x_1,\ldots, x_n\in E.
$$
A map $H: U\to F$ is said to be \emph{holomorphic} on a nonempty open subset $U$ of $E$
if for each $a $ in $U$ there exist an open ball $B_E(a;r) \subseteq U$,  centered at $a$ with radius $r>0$, and
a unique sequence of bounded $n$-homogeneous polynomials $P_n:E \to F$ such that
$$
H(x) = \sum_{n=0}^\infty P_n(x - a)
$$
uniformly for all $x $ in $B_E(a;r)$.

To study holomorphic functions,
we might assume, after translation, $a=0$.  A holomorphic function $H: B_E(0;r)\to F$ has its Taylor series at zero:
\begin{align}\label{eq:Taylor}
H(x) = \sum_{n=0}^\infty P_n(x)
\end{align}
uniformly for all $x$ in $B_E(0;r)$.
In the complex case, we have the Cauchy integral formulae:
\begin{align}\label{eq:cauchy}
P_n(x) = \frac{1}{2\pi i}\int_{|\lambda| = 1} \frac{H(\lambda x)}{\lambda^{n+1}}\,d\lambda, \qquad n = 0, 1, 2, \ldots.
\end{align}
For the general theory of homogeneous polynomials and holomorphic functions, we refer to \cite{Di, Mu}.

When $E, F$ are Banach  algebras,
a  function $\Phi:E\to F$  is said to be {\it
orthogonally additive} if
$$
fg = gf=0\quad\text{implies}\quad \Phi(f + g) = \Phi(f) + \Phi(g),\quad \forall f,g\in E,
$$
and \emph{orthogonally multiplicative} if
$$
fg = gf=0\quad\text{implies}\quad \Phi(f)\Phi(g) = 0,\quad \forall f,g\in E.
$$
The notions of orthogonally additive and orthogonally multiplicative transformations
have been studied by many authors, for example, \cite{Ar83,Ja90,JW96,HS99,CKLW03,araujo04,KLW04,PV,BLL06,CLZ06,PPV08,CLZ10,BR11,PP12,LW13}.

Our goal is to study orthogonally additive and
orthogonally multiplicative holomorphic functions between C*-algebras.
 Every abelian C*-algebra is the algebra $C_0(X)$ of continuous functions
of a locally compact Hausdorff space $X$ vanishing at infinity.
In general, a C*-algebra can be embedded into  $B(H)$ as
a norm closed self-adjoint subalgebra.
When $E,F$ are  algebras of continuous functions,
it is  established in \cite{BHW} the following nice representation.

\begin{proposition}[{\cite{BHW}}]
\label{prop:BSTPdphf}
Let $H: B_{C_0(X)}(0;r)\to C_0(Y)$ be a
bounded orthogonally additive and orthogonally multiplicative holomorphic function.
Then there exist a  sequence $\{h_n\}$ of bounded scalar continuous functions in $C(Y)$ and a map
$\varphi: Y\rightarrow X$ such that
$$
H(f)(y) = \sum_{n\geq1} h_n(y)(f(\varphi(y)))^n, \quad\forall y\in Y,
$$
uniformly for all $f $ in $B_{C_0(X)}(0;r)$.
Here,  $\varphi$ is continuous  wherever any $h_n$ is nonvanishing.
\end{proposition}

Talking  orthogonality of a pair of elements $a,b$ in a general C*-algebra (in $B(H)$), people usually refers one of the following situations.
\begin{enumerate}[(i)]
    \item Zero product: $ab =0$.
    \item Two side zero product: $ab=ba=0$.
    \item Range orthogonality: $a^*b=0$.
    \item Initial space or domain orthogonality: $ab^*=0$.
    \item Range and domain orthogonality: $a^*b = ab^* =0$.
\end{enumerate}
In the abelian case, however, all these concepts coincide.
They coincide in general when both $a,b$ are self-adjoint.

Some partial results concerning the structures of homogeneous polynomials  between
general C*-algebras are also given in \cite{BHW}.  For example, we have

\begin{proposition}[{\cite{BHW}}]\label{prop:BH}
Let $H$ be a complex Hilbert space of arbitrary dimension.
Let $P:B(H)\to B(H)$ be a bounded $n$-homogeneous polynomial, which is  additive and   multiplicative on pairs of orthogonal self-adjoint elements.
Suppose that $P(1)$ is invertible or
$P(B(H))\supset B(H)^+$.
Then there is a nonzero scalar $\lambda$ and
     an invertible operator $S$ in $B(H)$ such that either
    $$
    P(x) = \lambda S^{-1}x^nS,\ \forall x\in B(H),
    $$
or
   $$
    P(x) = \lambda S^{-1}(x^t)^nS,\ \forall x\in B(H).
    $$
\end{proposition}
Here, $x^t$ is the transpose of a bounded linear operator $x$ in $B(H)$ with respect to some arbitrary but fixed
orthogonal basis of the Hilbert space $H$.
For a matrix $x=(x_{ij})$, we simply define $x^t=(x_{ji})$ to be the transpose of $x$.

However,   results in \cite{BHW} usually assume a
rather strong hypothesis that $P(1)$ is invertible or $P(A)\supset B^+$.  It is not very likely every summand $P_n$ in the Taylor series
\eqref{eq:Taylor}
of a holomorphic function $H:B_A(0;r)\to B$ would satisfy one of these conditions.  Thus, a general structure result
about such holomorphic functions is still far away from reaching.

In this paper, we will establish another important case.
We will give a description of orthogonally additive and orthogonally multiplicative holomorphic function
$H: M_m\to M_m$ of complex matrix algebras.

In the following we say that a map $H$ between complex matrices is \emph{orthogonally additive} (resp.\ \emph{multiplicative})
\emph{on self-adjoint elements} if
$$
H(a+b) = H(a) + H(b)
$$
(resp.
$$
H(a)H(b) = 0)
$$
whenever $a,b$ are self-adjoint complex matrices in its domain with $ab=0$.

\begin{theorem}\label{thm:main-matrix}
Let $m$ and $s$ be
positive integers with $m\geq 2$ and $m\geq s$.
Let $H: B_{M_m}(0;r)\to M_s$ be a holomorphic function between  complex matrix algebras.
Assume $H$ is orthogonally additive and orthogonally multiplicative on self-adjoint elements.
Then either
\begin{enumerate}[(A)]
    \item the range of $H$ consists of zero trace elements
(this case occurs whenever $s<m$), or
    \item
there exist a  scalar sequence $\{\lambda_n\}$ (some $\lambda_n$ can be zero) and an invertible $m\times m$
    matrix $S$ such that
    $$
    H(x) = \sum_{n\geq 1} \lambda_n S^{-1}x^nS, \quad\forall x\in B_{M_m}(0;r),\eqno{(\ddag)}
    $$
or
    $$
    H(x) = \sum_{n\geq 1} \lambda_n S^{-1}(x^t)^nS, \quad\forall x\in B_{M_m}(0;r).
    $$
\end{enumerate}
In the  case (B), we always have the representation ($\ddag$) when $H$ also preserves zero products, i.e.,
$$
ab =0 \quad\implies H(a)H(b) = 0, \qquad \forall a,b \in B_{M_m}(0;r).
$$
\end{theorem}

The proof of Theorem \ref{thm:main-matrix} will be given in the next section.
The following example shows that the exception case in Theorem \ref{thm:main-matrix} can occur
when $s=m$.

\begin{example}\label{ex:trival-multiplication}
Let $E_{ij}$ be the matrix unit with `1' at the $(i,j)$th entry and `0' elsewhere.
Consider the linear map
$T:M_2\to M_2$ defined by
$T(E_{11})= E_{12}$,
and $T(E_{ij})=0$ for all other ${i,j}$.  Then $T$ is an
orthogonally multiplicative, and linear (and thus holomorphic) map.
It is plain that the range of $T$ consists of nilpotent  matrices.
\end{example}

On the other hand, we can have other possibilities when  the range have larger dimension
than the domain.

\begin{example}\label{eg:kk2}
Consider $\theta:M_k\to M_{k+2}$ defined by
$$
\begin{pmatrix}a_{ij}\end{pmatrix}
\mapsto
\begin{pmatrix}
    0&a_{11}&a_{12}&\ldots&a_{1k}&0\\
    0&0&0&\ldots&0&a_{11}\\
    0&0&0&\ldots&0&a_{21}\\
    &\vdots&&\ddots&&\vdots\\
    0&0&0&\ldots&0&a_{k1}\\
        0&0&0&\ldots&0&0
    \end{pmatrix}.\quad
$$
Then $\theta$ is  linear (and thus holomorphic), and orthogonally multiplicative
on self-adjiont elements.  Note that
the range of $\theta$ does not have trivial multiplication, since
$\theta(E_{11})^2=E_{1,k+2}$.
However,  $\theta$ cannot be written as the form $c\varphi$
for any fixed element $c$ in $M_{k+2}$ and any homomorphism or anti-homomorphism
$\varphi:M_k\to M_{k+2}$.
Assume on the contrary that $\theta=c\vp$.  Then we arrive at a contradiction
\begin{align*}
E_{1,k+2}&=\theta(E_{11})^2=\theta(E_{11})c\vp(E_{11})
    =\theta(E_{11})c(\vp(E_{12})\vp(E_{21}))\\
    &=\theta(E_{11})(c\vp(E_{12}))\vp(E_{21}))
    =\theta(E_{11})\theta(E_{12})\vp(E_{21})
    =0\vp(E_{21})=0.
\end{align*}
\end{example}

\begin{example}\label{eg:kk2-nonzerotrace}
Consider $\phi:M_k\to M_{2k+2}$ defined by
$$
\begin{pmatrix}a_{ij}\end{pmatrix}
\mapsto
\begin{pmatrix}a_{ij}\end{pmatrix} \oplus \theta\begin{pmatrix}a_{ij}\end{pmatrix}.
$$
Here, $\theta$ is the map defined in Example \ref{eg:kk2}.
Again $\theta$ is  linear (and thus holomorphic), and orthogonally multiplicative on self-adjoint elements.
However, $\theta$
cannot be written in any form stated in Theorem \ref{thm:main-matrix}(B),
although its range contains elements of nonzero trace.
\end{example}

The infinite dimensional case can be more complicated.

\begin{example}\label{eg:infty}
Let $H$ be a separable infinite dimensional Hilbert space with an orthonormal basis $\{e_n: n=1,2,\ldots\}$.
\begin{enumerate}[(a)]
    \item Consider $\theta:B(H)\to B(H)$ defined by
$$
\begin{pmatrix}a_{ij}\end{pmatrix}
\mapsto
J^2\begin{pmatrix}a_{ij}\end{pmatrix}{J^*}^2 + \begin{pmatrix}
    0&a_{11}\\
    0&0
    \end{pmatrix}.
$$
Here, $J: B(H)\to B(H)$ is the unilateral shift operator sending $e_n$ to $e_{n+1}$ for $n=1,2,\ldots$.
Then $\theta$ is not of the standard form, while its range contains elements of nonzero trace.

     \item Let $E$ and $F$ be the isometries in $B(H)$ such that $E(e_n)=e_{2n}$ and $F(e_n)=e_{2n-1}$ for $n=1,2,\ldots$, respectively.
Define a holomorphic function $H:B(H)\to B(H)$ by
$$
H(a) = EaE^* + F(a^t)^2F^*, \quad\forall a\in B(H).
$$
Then $H$ is orthogonally additive and orthogonally multiplicative, but not zero product preserving.
(Readers can make up one preserving zero products easily.)
The range of $H$ contains the identity $H(1)=1$.  However, it cannot be written in any form stated in Theorem \ref{thm:main-matrix}(B).
\end{enumerate}
\end{example}

\section{The proofs}

We begin with an  observation.

\begin{lemma}\label{lem:ortho-sums}
Let $H: B_{E}(0;r)\to F$ be a holomorphic function between C*-algebras with Taylor series at zero
$H = \sum_{n=0}^\infty P_n$.
\begin{enumerate}[(a)]
\item If $H$ is orthogonally additive on self-adjoint elements then each $P_n$ is also orthogonally additive on self-adjoint elements.
\item If $H$ is orthogonally multiplicative on self-adjoint elements  then each $P_n$ is also orthogonally multiplicative on self-adjoint elements.
Indeed, for orthogonal self-adjoint elements $x,y$ in $B_{E}(0;r)$ we have
\begin{align*}
xy=0 \quad\implies\quad P_m(x)P_n(y)=0, \quad m,n =0,1,2,\ldots.
\end{align*}
\end{enumerate}
\end{lemma}

\begin{proof}
Let $\{x,y\}$ be an orthogonal pair of self-adjoint elements in $B_{E}(0;r)$.
Suppose first that $H$ is orthogonally additive.
For sufficiently small scalar $\alpha$, we have
\begin{align*}
&H(\alpha x+\alpha y) = \sum_n P_n(\alpha x + \alpha y) = \sum_n \alpha^n P_n( x + y)\\
 =\ &H(\alpha x) + H(\alpha y) = \sum_{n} (P_n(\alpha x) + P_n(\alpha y))= \sum_{n} \alpha^n (P_n( x) + P_n( y)).
\end{align*}
As $\alpha$ can be arbitrary (but small), we see that
$$
P_n(x+y) =P_n(x)+P_n(y), \quad n=0,1,2,\ldots.
$$

Suppose then $H$ is orthogonally multiplicative.
For sufficiently small scalars $\alpha, \beta$, we have
\begin{align*}
0 = H(\alpha x)H(\beta y) = \sum_{m,n} P_m(\alpha x)P_n(\beta y)
=  \sum_{m,n} \alpha^m\beta^n P_m( x)P_n( y).
\end{align*}
As $\alpha,\beta$ can be arbitrary (but small), we see that
$$
P_n(x)P_m(y) =0, \quad n,m=0,1,2,\ldots.
$$
\end{proof}

It follows from  Lemma \ref{lem:ortho-sums}
that if $H$ is orthogonally additive  then $P_0 = 0$.

The following linearization of orthogonally additive $n$-homogeneous
polynomials of matrix algebras is an important tool of us.
The  result for general C*-algebras is given by
C. Palazuelos, A. M. Peralta and I. Villanueva \cite{PPV08}, and
M. Burgosy, F. J. Fern\'{a}ndez-Poloz, J. J. Garc\'{e}sx and A. M. Peralta \cite{BFGP09},
which extend the
commutative version of D. Perez-Garcia and I. Villanueva \cite{PV} (see also \cite{PPV08}).

\begin{lemma}\label{lem:C-form}
Let  $F$ be a complex Banach space, and $P:M_m\to F$ an
$n$-homogeneous polynomial.   If $P$ is orthogonally additive on self-adjoint elements then
there exists a linear operator $T : M_m \to F$ such that
$$
P(x) = T(x^n), \quad\forall x \in M_m.
$$
\end{lemma}

Recall that we say a map $\theta$ between rings preserving zero products if $\theta(x)\theta(y)=0$
whenever $xy=0$.
We say that a set $Z$ of a ring has trivial products, if $xy=0$ for all $x,y$ in $Z$.

\begin{lemma}[{\cite[Corollary 2.4]{CKLW03}}]\label{lem:zero}
Let $m$ and $s$ be
positive integers with $m\geq 2$ and $m\geq s$.
Let ${\mathbb F}$ be an algebraically closed field of
characteristic $0$ and $\theta:M_m({\mathbb F})\to M_s({\mathbb
F})$ a linear map preserving zero products. Then either
the range of $\theta$ has trivial multiplication,
or $m=s$ and there exist an invertible matrix $S$ in $M_m(\mathbb F)$ and a nonzero scalar $c$ such that
$$
\theta(x)=cS^{-1}xS \quad \forall x\in M_m(\mathbb{F}).
$$
\end{lemma}

Note that the orthogonal multiplicity of an orthogonal additive polynomial $P$ does not guarantee its linearization
$T$ preserving zero products.  So we cannot apply Lemma \ref{lem:zero} directly.
But when $x,y$ are idempotents with $xy=0$, we have $T(x)T(y) = P(x)P(y)=0$.
This suggests we establish the following two lemmas.  Fortunately, they are sufficient for our proof of Theorem \ref{thm:main-matrix}.

\begin{lemma}\label{lem:nr}
Let $m$ and $s$ be
positive integers with $m\geq 2$ and $m\geq s$.
Let  $\theta:M_m\to M_s$ be a complex linear map.
Assume that
\begin{align}\label{eq:pq}
\theta(p)\theta(q) =0\quad\text{whenever}\quad \text{$p,q$ are orthogonal rank one projections.}
\end{align}
Then either
\begin{enumerate}[(A)]
    \item the range of $\theta$ consists of nilpotent elements (this happens whenever $s<m$),
or
    \item $m=s$ and there exist an invertible matrix $S$ in $M_m$ and a nonzero scalar $\lambda$ such that
$$
\theta(x)=\lambda S^{-1}xS, \quad \forall x\in M_m,
$$
or
$$
\theta(x)=\lambda S^{-1}x^tS, \quad \forall x\in M_m.
$$
\end{enumerate}
\end{lemma}
\begin{proof}
Note that \eqref{eq:pq} holds indeed for all orthogonal pairs  of self-adjoint matrices (through spectral decompositions).
For any projection $p$ in $M_m$, we have
$$
(1-p)p=p(1-p)=0.
$$
The assumption implies that
$$
\theta(1-p)\theta(p)=\theta(p)\theta(1-p)=0.
$$
Hence
$$
\theta(1)\theta(p)=\theta(p)\theta(1)=\theta(p)^2
$$
holds for all projections $p$ in $M_m$.  By the spectral theory, we have
\begin{align}\label{eq:1a}
\theta(1)\theta(a)=\theta(a)\theta(1)
\end{align}
for all self-adjoint, and thus for all, $a$ in $M_m$.
Moreover,
$$
\theta(1)\theta(a^2)=\theta(a^2)\theta(1)=\theta(a)^2
$$
holds for all self-adjoint elements $a$ in $M_m$.
Considering $(a+b)^2$ for two self-adjoint elements $a$ and $b$,
we have
$$
\theta(1)\theta(ab+ba)=\theta(a)\theta(b) + \theta(b)\theta(a).
$$
Since very complex matrix $a$ in $M_m$ can be written as $a=b+\sqrt{-1}c$  for two self-adjoint matrices $b$ and $c$,
we see that
\begin{align}\label{eq:1a2}
\theta(1)\theta(a^2)=\theta(a^2)\theta(1)=\theta(a)^2, \quad\forall a\in M_m.
\end{align}
It follows further that
\begin{align}\label{eq:1jordan}
\theta(1)\theta(ab + ba) = \theta(a)\theta(b) + \theta(b)\theta(a),\quad\forall a,b\in M_m.
\end{align}

Suppose that there is an $x$ in $M_m$ such that $\theta(x)$ is not nilpotent.
It follows from \eqref{eq:1a} and \eqref{eq:1a2} that
$$
\theta(1)^s\theta(x^{2})^s = (\theta(1)\theta(x^{2}))^s = \theta(x)^{2s} \neq 0.
$$
Consequently, $\theta(1)$ is not nilpotent, and thus its spectrum contains a complex number
$\lambda\neq 0$.  Using the Riesz functional calculus, we have
an idempotent $e$ in $M_s$ such that
$e\theta(1)=\theta(1)e= \lambda e\neq 0$ and $e$ commutes with every matrix commuting with $\theta(1)$
(see, e.g., \cite[Prop.\ 4.11]{Conway90}).
Since $\theta(1)$ commutes with all $\theta(a)$'s, so does $e$.
Define $\Psi: M_m\to M_s$ by $\Psi(a)=e\theta(a)/\lambda$.  It follows from \eqref{eq:1jordan} that
$$
\Psi(ab+ba)=\Psi(a)\Psi(b) + \Psi(b)\Psi(a), \quad\forall a, b\in M_m.
$$
By the well-known theorem of Herstein, we see that either
$\Psi=0$, or $\Psi$ is an injective
homomorphism or anti-homomorphism.
But the first case implies the contradiction $e = \Psi(1)=0$.
Hence, the latter case occurs, and  we must have $s=m$.
This forces $e=1$ and thus $\theta(1)=\lambda$.
It follows from the Noether-Skolem theorem that we have one of the expected representations of $\theta$.
 This completes the proof.
\end{proof}

\begin{lemma}\label{lem:idem-ortho}
Let $m$ and $s$ be
positive integers with $m\geq 2$ and $m\geq s$.
Let  $\theta:M_m\to M_s$ be a complex linear map.
Assume that
\begin{align*}
\theta(e)\theta(f) =0\quad\text{whenever}\quad \text{$e,f$ are rank one idempotents with $ef=0$.}
\end{align*}
If the range of $\theta$ does not consist of nilpotent elements then $s=m$,
and there exist an invertible matrix $S$ in $M_m$ and a nonzero scalar $\lambda$ such that
$$
\theta(x)=\lambda S^{-1}xS, \quad \forall x\in M_n.
$$
\end{lemma}
\begin{proof}
Since projections are idempotents, it follows from Lemma \ref{lem:nr} that $m=s$ and
there is an invertible matrix $S$ and a nonzero scalar $\lambda$ such that
$\frac{1}{\lambda} S\theta(x)S^{-1}$ is either always $x$ or always $x^t$.
 Consider the idempotents
$a,b$ in $M_m$ with the top left $2\times 2$ blocks given below and zero elsewhere, respectively:
$$
\left(
  \begin{array}{cc}
    1 & 0 \\
    0 & 0 \\
  \end{array}
\right)
\quad\text{and}\quad
\left(
  \begin{array}{cc}
    0 & 0 \\
    1 & 1 \\
  \end{array}
\right).
$$
Then $ab=0$ but $ba \neq 0$.  Since $\theta$ preserves zero products, so does $\frac{1}{\lambda} S\theta S^{-1}$.
Because $a^t b^t = (ba)^t\neq 0$, we conclude that $\frac{1}{\lambda} S\theta(x)S^{-1}=x$ for all $x$ in $M_m$.
This gives us the desired assertion.
\end{proof}

\begin{proof}[{Proof of Theorem \ref{thm:main-matrix}}]
It is not difficult to see  that the two cases stated in the conclusions are exclusive.
Assume from now on $H(d)$ is a matrix in $M_s$ of nonzero trace for some $d$ in $B_{M_m}(0;r)$.

Lemma \ref{lem:ortho-sums} ensures that each summand $P_n$ is an orthogonally additive and orthogonally multiplicative
$n$-homogeneous polynomial, and the constant term $P_0=H(0)=0$.
Lemma \ref{lem:C-form} provides a
linear map $T_n: M_m\to M_s$ for each $n$ such that
$$
P_n(x)=T_n(x^n), \quad\forall x\in M_m.
$$
Inherited from
$\{P_n\}$, the family $\{T_n\}$  satisfies the orthogonality preserving property stated in \eqref{eq:pq}.

Since $H(d) = \sum_n T_n(d^n)$, the continuity of the trace functional ensures that
some
$T_k(d^k)$ in the sum has nonzero trace.  In particular, $T_k(d^k)$ is not a nilpotent.
Lemma \ref{lem:nr} ensures $m=s$ and provides an invertible
matrix $S_k$ and a nonzero scalar $\lambda_k$   such that
$$
P_k(x) = \lambda_k S_k^{-1}x^kS_k, \quad\forall x\in B_{M_m}(0;r),
$$
or
$$
P_k(x) = \lambda_k S_k^{-1}(x^t)^kS_k, \quad\forall x\in B_{M_m}(0;r).
$$

We claim that all other $P_n$ either carries a similar form or constantly zero.
Redefining $H(x)$ with $\frac{1}{\lambda_k}S_kH(x)S_k^{-1}$ or  $\frac{1}{\lambda_k}S_kH(x^t)S_k^{-1}$, we can assume
$$
P_k(x) = x^k, \quad\forall x\in B_{M_m}(0;r).
$$
Suppose that with a nonzero scalar $\lambda_n$ and an invertible $S_n$ in $M_m$ we have
$$
P_n(x)=\lambda_n S_n^{-1}x^nS_n, \quad\forall x\in B_{M_m}(0;r),
$$
or
$$
P_n(x)=\lambda_n S_n^{-1}(x^t)^nS_n, \quad\forall x\in B_{M_m}(0;r).
$$
By Lemma \ref{lem:ortho-sums}(b),
$$
xS_n^{-1}yS_n = S_n^{-1}yS_n x = 0, \quad\text{whenever $x,y$ are orthogonal  projections}.
$$
This forces
$$
S_n^{-1}yS_n = \alpha_y y
$$
with some scalar $\alpha_y$ for every rank one projection $y$ in $M_m$.
In particular, every nonzero vector is an eigenvector of $S_n$.
Thus $S_n =\alpha I$ for some nonzero scalar $\alpha$.
In other words, we can assume  that $P_n(x)=\lambda_n x^n$  for all $x$ in $B_{M_m}(0;r)$,
or $P_n(x)=\lambda_n (x^t)^n$  for all $x$ in $B_{M_m}(0;r)$.  However, Lemma \ref{lem:ortho-sums}(b) ruins out
the possibility of the second case.  For example, try the pair of orthogonal projections
$$
a=\frac{1}{2}\left(
  \begin{array}{cc}
    1 & \sqrt{-1} \\
    -\sqrt{-1} & 1 \\
  \end{array}
\right)\oplus \mbox{\large $0$}
\quad\text{and}\quad
b= \frac{1}{2}\left(
  \begin{array}{cc}
    1 & -\sqrt{-1} \\
    \sqrt{-1} & 1 \\
  \end{array}
\right)\oplus \mbox{\large $0$}.
$$
While $ab=0$, we have $ab^t\neq 0$.

Suppose next that there is a $P_n$ whose range consists of nilpotent elements.  We will verify that $P_n=0$.
Arguing as above, we have
$$
xT_n(y)= T_n(y)x =0, \quad\text{whenever $x,y$ are orthogonal  projections}.
$$
This forces
$$
T_n(y)=\alpha_y y
$$
with some scalar $\alpha_y$ for every rank one projection $y$ in $M_m$.
Since $\beta^{mn}T_n(y)^m= P_n(\beta y)^m=0$ for all small scalar $\beta >0$, we see that $\alpha_y=0$.
Since every self-adjoint  matrix is an orthogonal sum of rank one projections,
the linear map $T_n=0$, and thus $P_n=0$,  on $M_m$.

The claim is established.
It follows  that there is a scalar sequence $\{\lambda_n\}$  and an invertible $S$ in $M_m$ such that
$$
H(x) = \sum_{n\geq 1} \lambda_n S^{-1}x^nS,\quad\forall x\in B_{M_m}(0;r).
$$
Translating back to the original situation, there is also another possible case that
$$
H(x) = \sum_{n\geq 1} \lambda_n S^{-1}(x^t)^nS,\quad\forall x\in B_{M_m}(0;r).
$$

Finally, assume that $H$ also preserves zero products, and thus so does
every $P_k$.
Consequently, the linearization $T_k$ sends two rank one idempotents with zero products
to a pair of elements with zero products.
By Lemma \ref{lem:idem-ortho}, we have $T_k(x) = \lambda_k S^{-1}x^k S$ for all $x$ in $M_m$.
This forces $H$ carries the first form as in ($\ddag$).
This completes the proof.
\end{proof}

\bigskip

\noindent
\textbf{Acknowledgement.}
This research is supported partially by the Taiwan NSC grant 102-2115-M-110-002-MY2.

The authors are grateful to the referee for his/her careful reading and critical comments.
In particular, the referee questions about the application of Lemma \ref{lem:zero} and a statement in a previous version of this note
concerning the validity of our arguments
in the case when the underlying field is the real numbers.  Due to lack of sufficient tools, we are unable to extend
the results in this note to the real case
at this moment.  We hope to finish this task in a subsequent project.

\bigskip

\noindent
\textbf{Added in proofs.}
In a very recent paper \cite{GPPI13},  J. J. Garc\'es,   A. M. Peralta,   D. Puglisi and  R. M. Isabel obtain another satisfactory result about orthogonally additive holomorphic functions $H:B_A(0;r)\to B$, where $A,B$ are general C*-algebras.
Assume that the range of $H$ contains an invertible element, and
$$
H(a)^*H(b)=H(a)H(b)^*=0\quad\text{whenever}\quad ab=0 \text{ for  self-adjoint } a,b \in B_A(0,r).
$$
They prove that there exist a sequence $\{h_n\}$ in  $B^{**}$ and
Jordan *-homomorphisms $J,\tilde J$ from the multiplier algebra $M(A)$ of $A$ into $B^{**}$ such that
$$
H(a) = \sum_{n\geq 1} h_n J(a^n) = \sum_{n\geq1} \tilde J(a^n)h_n, \quad\forall a\in B_A(0;r).
$$

\end{document}